\documentclass[12pt,reqno]{amsart}
\usepackage[final]{graphicx}
\usepackage{amsfonts}
\usepackage{pdfsync}
\usepackage{amsmath}
\usepackage{amssymb}
\usepackage{amsthm}

\usepackage{amsbsy}
\usepackage{enumerate}
\usepackage[T1]{fontenc}
\usepackage[utf8]{inputenc}
\usepackage{lmodern}

\usepackage{comment}

\newtheorem{theorem}{Theorem}[section]
\newtheorem{lemma}[theorem]{Lemma}

\newtheorem{remark}[theorem]{Remark}

\newtheorem{claim}[theorem]{Claim}
\newtheorem{definition}{Definition}
\DeclareMathOperator{\arctanh}{arctanh}

\def\vf{\varphi}

\def\t{\theta}

\begin{document}

\title{\textbf{The solution gap of the Brezis--Nirenberg Problem on the hyperbolic space}}

\author[Benguria]{Soledad Benguria$^1$}

\address{$^1$ Department of Mathematics, University of Wisconsin - Madison}


\smallskip

\begin{abstract}
We consider the positive solutions of the nonlinear eigenvalue problem $-\Delta_{\mathbb{H}^n} u = \lambda u + u^p, $ with $p=\frac{n+2}{n-2}$ and $u \in H_0^1(\Omega),$ where $\Omega$ is a geodesic ball of radius $\theta_1$ on $\mathbb{H}^n.$ For radial solutions, this equation can be written as an ODE having $n$ as a parameter. In this setting, the problem can be extended to consider real values of $n.$ We show that if $2<n<4$ this problem has a unique positive solution if and only if  $\lambda\in \left( n(n-2)/4 +L^*\,,\, \lambda_1\right).$ Here $L^*$ is the first positive value of $L = -\ell(\ell+1)$ for which a suitably defined associated Legendre function $P_{\ell}^{-\alpha}(\cosh\theta) >0$ if $0 < \t<\t_1$ and $P_{\ell}^{-\alpha}(\cosh\t_1)=0,$ with $\alpha = (2-n)/2.$

\end{abstract}


\maketitle
\section{Introduction}\label{intro}

Given a bounded domain $\Omega$ in $\mathbb{R}^n,$ Brezis and Nirenberg  \cite{BrNi83} considered the problem of existence of a function $u \in H_{0}^1(\Omega)$ satisfying 

\begin{eqnarray}\label{eq:BrNioriginal}
-\Delta u &=& \lambda u + u^p \,\,\,\, \text{on} \,\,\, \Omega \\
u &>& 0\,\,\,\,\,\,\,\,\, \qquad  \text{on}\,\,\, \Omega  \nonumber\\
u & = & 0 \,\,\,\,\,\,\,\,\, \qquad  \text{on}\,\,\, \partial\Omega,  \nonumber
\end{eqnarray}

\noindent where $p= (n+2)/(n-2)$ is the critical Sobolev exponent. If $\lambda \ge \lambda_1,$ where $\lambda_1$ is the first Dirichlet eigenvalue, this problem has no solutions. Moreover, if the domain is star-shaped, there is no solution if $\lambda \le 0.$ Thus, when $\Omega$ is a ball, for any given value of $n$ there may exist a solution only if $\lambda\in(0,\lambda_1).$ It was shown in \cite{BrNi83} that in dimension $n \ge 4,$ there exists a solution for all $\lambda$ in this range. However, in dimension $n=3$ Brezis and Nirenberg showed there is an additional interval where there is no solution, which we will refer to in this article as the {\it solution gap} of the Brezis-Nirenberg problem. When the domain is the unit ball, the solution gap when $n=3$ is the interval $\left(0, \frac{\lambda_1}{4}\right].$   

\bigskip

The dimensions for which semilinear second order elliptic problems with a nonlinear term of critical exponent (of which (\ref{eq:BrNioriginal}) is an example) have a solution gap are referred to in the literature as critical dimensions. This definition was first introduced by Pucci and Serrin in \cite{PuSe90}. In \cite{Ja99}, Jannelli studies a general class of such problems, and the associated critical dimensions. He gives an alternative proof to the existence results obtained in \cite{BrNi83} for problem (\ref{eq:BrNioriginal}). When $\Omega$ is a ball, and $n=3,$ Jannelli shows that (\ref{eq:BrNioriginal}) has no solution if $\lambda \le j_{\alpha,1}^2,$ where $\alpha = (2-n)/2$ and $j_{\alpha,1}$ denotes the first positive zero of the Bessel function $J_{\alpha}.$

\bigskip

If $u$ is radial, problem (\ref{eq:BrNioriginal}) can be written as an ordinary differential equation,
$$ -u'' -\frac{(n-1)}{r}u' = \lambda u + u^p,$$
\noindent  where $n$ can be thought of as a parameter in the equation, rather than the dimension of the space. By doing so one can study the behavior of the solution gap with respect to $n$ by taking $n$ to be a real number instead of a natural number. Jannelli's methods in \cite{Ja99} can be easily extended to the case $2<n<4,$ thus concluding that the solution gap of the Brezis-Nirenberg problem defined in the unit ball is the interval $\left(0, j_{\alpha,1}^2\right].$ In particular, it follows that $n=4$ is the first value of $n$ for which there is no solution gap. 

\bigskip

The solution gap of the Brezis-Nirenberg problem can also be studied in non-Euclidean spaces. On the sphere $\mathbb{S}^n,$ for a fixed $n,$ the solution gap is the subinterval of $(-n(n-2)/4,\lambda_1)$ for which (\ref{eq:BrNioriginal}) has no solution. As in the Euclidean case, $n=3$ is a critical dimension, whereas $n\ge 4$ are not. It was shown in \cite{BaBe02} that if $\Omega$ is a geodesic cap of radius $\theta_1$ in $\mathbb{S}^3$ the solution gap is the interval $(-n(n-2)/4,(\pi^2-4\theta_1^2)/4\theta_1^2].$ If $u$ is radial, then (\ref{eq:BrNioriginal}) can be written as an ordinary differential equation that still makes sense when $n$ is a real number. It was shown in \cite{BeBe15} that if $2<n<4,$ the solution gap is the interval $ \left( -n(n-2)/4, \left((2\ell^*+1)^2 - (n-1)^2\right)/4\right],$ where $\ell^*$ is the first positive value of $\ell$ for which the associated Legendre function $P_{\ell}^\alpha(\cos \theta_1)$ vanishes.  Here  $\alpha = (2-n)/2.$


\bigskip

In this article we consider the Brezis-Nirenberg problem on the hyperbolic space $\mathbb{H}^n.$  That is, we consider the problem 

\begin{eqnarray}\label{eq:BrNihiper}
-\Delta_{\mathbb{H}^n} u &=& \lambda u + u^p \,\,\,\, \text{on} \,\,\, \Omega \\
u &>& 0\,\,\,\,\,\,\,\,\, \qquad  \text{on}\,\,\, \Omega  \nonumber\\
u & = & 0 \,\,\,\,\,\,\,\,\, \qquad  \text{on}\,\,\, \partial\Omega,  \nonumber
\end{eqnarray}

\noindent where $p= (n+2)/(n-2),$ $\Omega$ is a geodesic ball on $\mathbb{H}^n$ of radius $\theta_1\in(0,\infty),$ and $\Delta_{\mathbb{H}^n}$ is the Laplace-Beltrami operator. 

\bigskip

It is not hard to show (see, e.g., page 285 in \cite{St02}) that there can be no solutions for $\lambda \not\in \left( n(n-2)/4, \lambda_1\right).$ Stapelkamp \cite{St02} showed that if $n\ge 4$ there is no solution gap, that is, that there is a solution for all values of $\lambda$ in this interval. When $n=3,$ however, she showed there is no solution if $ \lambda \in\left( n(n-2)/4, \lambda^*\right].$ Here 

$$ \lambda^* = 1 + \frac{\pi^2}{16 \arctanh^2 R}, $$

\noindent where $R$ is the radius of the ball that results by taking the stereographic projection of the geodesic ball onto $\mathbb{R}^3.$ Moreover, Stapelkamp shows that for each $\lambda \in (\lambda^*,\lambda_1),$ there exists a unique solution, and this solution is radial. A full characterization of the solutions to this problem in dimension $n\in \mathbb{N}$ (and any $p>1$) is given in \cite{BaKa12}. After the results of Stapelkamp and Bandle, there has been a vast literature on Brezis-Nirenberg type equations on hyperbolic spaces (see, e.g., \cite{MaSa08}, \cite{GaSa14}, \cite{GaSa15}, \cite{BoGaGrVa13}).   

\bigskip

For radial functions $u,$ problem (\ref{eq:BrNihiper}) can be written as an ordinary differential equation, with $n$ now simply representing a parameter in the equation rather than the dimension of the space. Our main result is that the solution gap of the Brezis-Nirenberg problem on the hyperbolic space has width $L^*,$ where $L^*$  is the first positive value of $L= -\ell(\ell+1)$ for which a suitably defined (see equation (\ref{eq:expansionenseries})) associated Legendre function $P_{\ell}^{-\alpha}(\cosh\theta) $ is positive if $0 < \theta<\theta_1$ and $P_{\ell}^{-\alpha}(\cosh\theta_1)=0.$ Here, as before, $\alpha = (2-n)/2.$ 


\bigskip 

More precisely, we show the following:

\begin{theorem}\label{main}
For any $2<n<4$ and $\theta_1\in(0,\infty),$ the boundary value problem
\begin{equation}\label{eq:ode}
-u''(\theta) - (n-1)\coth \theta\, u'(\theta) = \lambda u+ u^{\frac{n+2}{n-2}}
\end{equation}
with $u\in H_0^1(\Omega),$ $u'(0) = u(\theta_1)=0,$ and $\theta\in  [0,\theta_1]$ has a unique positive solution if and only if
\begin{equation}\label{eq:intervalo}
   \lambda\in \left( \frac{n(n-2)}{4} +L^*\,,\, \lambda_1\right).
   \end{equation}
\end{theorem}

In Figure \ref{figura} the graph $\lambda(n)$ illustrates the results of Theorem \ref{main}  when $\theta_1=1.$ The shaded region represents the solution gap, and the region between the dotted and the solid lines corresponds to the region of existence of solutions given by (\ref{eq:intervalo}). 
\bigskip
\begin{figure}[h]
    \centering
    \includegraphics[width=0.6\textwidth]{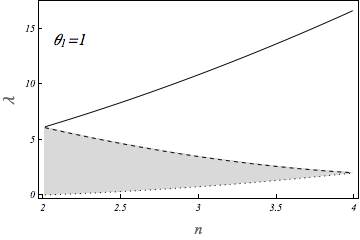}
    \label{figura}
\caption{The shaded region depicts the solution gap of the Brezis-Nirenberg problem in the hyperbolic space. The solid line corresponds to $\lambda_1,$ the dashed line to $\lambda =  n(n-2)/4 + L^*,$ and the dotted line to $\lambda = n(n-2)/4.$  }
\end{figure}

In Section \ref{sec:preliminaries} we derive an expression for the first Dirichlet eigenvalue in terms of the parameter $\ell$ of an associated Legendre function, and use this expression to show that the interval of existence given by (\ref{eq:intervalo}) is non-empty if $2<n<4.$ In Section \ref{sec:existencia} we use a classical Lieb lemma to show the existence of solutions for $\lambda$ as in  (\ref{eq:intervalo}). In Section \ref{sec:noexistencia} we use a Pohozaev type argument to show that if $2<n<4$ there is a solution gap of the Brezis-Nirenberg problem. That is, we show there are no solutions if $\lambda \in \left( n(n-2)/4\,,\, n(n-2)/4 +L^*\right].$ Finally, in Section \ref{sec:unicidad} we show that the uniqueness of solutions follows directly from \cite{KwLi92}.


\section{Preliminaries}\label{sec:preliminaries}

The associated Legendre functions $P_{\ell}^{\alpha}(\cosh\theta)$ and $P_{\ell}^{-\alpha}(\cosh\theta)$ are solutions of the Legendre equation

\begin{equation}\label{eq:ecndeP}
 y''(\theta) +\coth\theta\, y'(\theta)+ \left( -\ell(\ell+1) - \frac{\alpha^2}{\sinh^2\theta}\right)y(\theta) = 0.
\end{equation}

We will adopt the following convention for the associated Legendre functions:

\begin{equation}\label{eq:expansionenseries}
P_{\ell}^{\alpha}(\cosh\theta)  = \frac{1}{\Gamma(1-\alpha)}\coth^{\alpha}\left(\frac{\theta}{2}\right) {}_2F_1\left[-\ell,\ell+1,1-\alpha;-\sinh^2\left(\frac{\theta}{2}\right)\right],
\end{equation}

\noindent where for complex numbers $a,$ $b,$ and $c,$ the hypergeometric function ${}_2F_1[a,b,c;z]$ is given by

\begin{equation}\label{eq:hg}
{}_2F_1[a,b,c;z] =\sum_{n=0}^{\infty} \frac{(a)_n(b)_n}{(c)_n}\frac{z^n}{n!},
\end{equation}

\noindent where $ (\beta)_n := \Pi_{j=0}^{n-1}(\beta+j),$ for $\beta \in \mathbb{C}.$  

\begin{remark}
Notice that the associated Legendre functions $P_{\ell}^{\alpha}(\cosh\theta)$ depend on $\ell$ through the product $\ell(\ell+1),$ rather than from $\ell$ and $\ell+1$ independently.  \end{remark}

\noindent The associated Legendre functions given by (\ref{eq:ecndeP}) satisfy the following raising and lowering relations (see, e.g., \cite{Ri81}, page 55, equations (20.11-1) and (20.11-2) with $x=\cosh \theta$):

\begin{equation}\label{eq:subida}
\dot P_{\ell}^{\alpha}(\cosh\theta) = \frac{1}{\sinh\theta}P_{\ell}^{\alpha+1}(\cosh\theta) + \frac{\alpha\cosh\theta}{\sinh^2\theta}P_{\ell}^{\alpha}(\cosh\theta),
\end{equation}
\noindent and
\begin{equation}\label{eq:bajada}
\dot P_{\ell}^{\alpha+1}(\cosh\theta) = \frac{\ell(\ell+1)-\alpha(\alpha+1)}{\sinh\theta}P_\ell^{\alpha}(\cosh\theta)- \frac{(\alpha+1)\cosh\theta}{\sinh^2\theta}P_{\ell}^{\alpha+1}(\cosh\theta).
\end{equation}

\noindent Here $\dot P_\ell^{\alpha}$ means the derivative of $P_{\ell}^{\alpha}$ with respect to its argument. That is, 

$$\frac{d}{d\theta}P_{\ell}^{\alpha}(\cosh \theta) = \sinh \theta \dot P_{\ell}^{\alpha}(\cosh \theta).$$
\noindent Equations (\ref{eq:subida}) and (\ref{eq:bajada}) are used in the proof of the non-existence result on Section \ref{sec:noexistencia}.


\bigskip 

\begin{definition}\label{def:losl}
Let $L = -\ell(\ell+1).$ For $2<n<4,$ $\alpha = (2-n)/2,$ and $  \theta_1 \in (0,\infty),$ let  $L_1 $ be the smallest positive value of $L$ such that $P_{\ell}^{\alpha}(\cosh\theta)>0$ if $0 < \theta<\theta_1$ and $P_{\ell}^{\alpha}(\cosh\theta_1)=0.$ Similarly, let $L^*$ be the smallest positive value of $L$ such that $P_{\ell}^{-\alpha}(\cosh\theta) >0$ if $0 < \theta<\theta_1$ and $P_{\ell}^{-\alpha}(\cosh\theta_1)=0.$

\end{definition}

\bigskip 

In the next lemma we derive an expression for the first Dirichlet eigenvalue of $-\Delta_{\mathbb{H}^n}u = \lambda u$ on a geodesic ball in terms of $L_1.$ In Lemma \ref{cl:nonempty} we use the expression for $\lambda_1$ obtained in Lemma \ref{cl:lambda1} to show that the interval of existence given in equation (\ref{eq:intervalo}) is non-empty if $2<n<4.$


\begin{lemma}\label{cl:lambda1}

The first Dirichlet eigenvalue of equation 

\begin{equation}\label{eq:Dirichletoriginal}
-u'' - (n-1)\coth \theta u' = \lambda_1 u. 
\end{equation}

\noindent is given by

 $$ \lambda_1 =   \frac{n(n-2)}{4} +L_1.$$
\end{lemma}


\begin{proof}

\noindent Making the change of variables $u(\theta) = \sinh^{\alpha}\theta v(\theta),$ we can write equation (\ref{eq:Dirichletoriginal}) as

$$ v''(\theta)+(2\alpha \coth\theta +(n-1)\coth\theta)v'(\theta)+ (\alpha(\alpha+n-2)\coth^2\theta +\alpha +\lambda_1)v(\theta) = 0.$$

\noindent Choosing $\alpha = \frac{2-n}{2},$ one obtains 

$$ v''(\theta) + \coth\theta \, v'(\theta) + (\alpha+\lambda_1 - \alpha^2\coth^2\theta)v(\theta) = 0.$$

\noindent That is,

 $$ v''(\theta) + \coth\theta \, v'(\theta) + \left(\lambda_1-\alpha(\alpha-1) - \frac{\alpha^2}{\sinh^2\theta}\right)v(\theta) = 0.$$
 
 The solutions to this equation are $P_\ell^{\alpha}(\cosh\theta)$ and $P_{\ell}^{-\alpha}(\cosh\theta),$  where $\ell(\ell+1) = \alpha(\alpha-1)-\lambda_1.$ Since $\alpha$ is negative if $2<n<4,$ the regular solution of (\ref{eq:Dirichletoriginal}) is 
 
 $$ u(\theta) =  \sinh^{\alpha}\theta P_\ell^{\alpha}(\cosh\theta).$$ 
 
 
 \noindent  To satisfy the boundary condition $u(\theta_1) = 0,$ while having $u(\theta)>0$ in $(0,\theta_1),$ we must choose $\ell$ such that $-\ell(\ell+1)=L_1.$ Thus, 
 
 $$ \lambda_1 =   \frac{n(n-2)}{4} + L_1.$$
 
 \end{proof}
 
 
 \begin{remark}
It is known by \cite{Mc70} that $\lambda_1 \ge \frac{(n-1)^2}{4}.$ Thus, $-L_1  \le \frac{n(n-2)}{4} - \frac{(n-1)^2}{4}  = -\frac{1}{4}.$ 

 \end{remark}
 
 \bigskip 
 

\begin{lemma}\label{cl:nonempty}
 Let $L_1$ and $L^*$ be as in Definition \ref{def:losl}. Then $L^*<L_1.$ 
\end{lemma}


\begin{proof}
Let $y_1(\theta) = P_{\ell_1}^{\alpha}(\cosh\theta),$ and $y_2(\theta) = P_{\ell^*}^{-\alpha}(\cosh\theta). $ Then $y_j,$ $j\in\{1,2\},$ satisfy  

\begin{equation}\label{eq:y12}  y_j'' + \coth\theta  y_j' + k_j y_j = 0, 
\end{equation}
\noindent where 
\begin{equation*}
k_1 = L_1 - \frac{\alpha^2}{\sinh^2\theta}.
\end{equation*}

\noindent and

\begin{equation*}
k_2 = L^* - \frac{\alpha^2}{\sinh^2\theta}.
\end{equation*}

Let $W= y_1'y_2 -  y_2' y_1 $ and $W' = y_1'' y_2 - y_1  y_2''.$ Then it follows from equation (\ref{eq:y12}) that 
$$  W' +\coth\theta\, W  = (k_2-k_1) y_1y_2.$$ 

\noindent Multiplying by $\sinh\theta$ and integrating one has that 

$$ \int_0^{\theta_1} (W\sinh\theta)'\,d\theta = \left[ L^*-L_1\right] \int_0^{\theta_1} y_1 y_2\sinh\theta\,d\theta. $$ 

\noindent By choice of $L_1$ and $L^*$ it follows that $y_1$ and $y_2$ are positive on $[0,\theta_1)$ and vanish at $\theta_1,$ so that $\displaystyle\int_0^{\theta_1} y_1 y_2\sinh\theta\,d\theta$ is positive and $W(\theta_1) = 0.$ Thus, it suffices to show that $\lim_{\theta\rightarrow 0 }W(\theta)\sinh\theta$ is negative. 

\bigskip

It follows from equation (\ref{eq:expansionenseries}) that the behavior of $y_1$ and $y_2$ near zero is 

$$ y_1 \approx \frac{1}{\Gamma(1-\alpha)}\coth^{\alpha}\left( \frac{\theta}{2}\right),$$

\noindent and 

$$ y_2 \approx \frac{1}{\Gamma(1+\alpha)}\coth^{-\alpha}\left( \frac{\theta}{2}\right).$$

\noindent Therefore, 

\begin{equation*}\begin{split}
\lim_{\theta\rightarrow 0 }W(\theta)\sinh\theta =&\, \frac{-\alpha}{\Gamma(1-\alpha)\Gamma(1+\alpha)}\lim_{\theta\rightarrow 0} \sinh\theta \left(\frac{\tanh\left(\frac{\theta}{2}\right)}{\sinh^2\left(\frac{\theta}{2}\right)} \right) \\= &\,\frac{-2\alpha}{\Gamma(1-\alpha)\Gamma(1+\alpha).}
\end{split}\end{equation*}

\noindent Finally, since $\Gamma(1+\alpha)= \alpha\Gamma(\alpha),$ $\Gamma(\alpha)\Gamma(1-\alpha) = \pi \sin^{-1}(\pi\alpha),$  and $0<\alpha<1,$ we conclude that  

$$\lim_{\theta\rightarrow 0 }W(\theta)\sinh\theta =  \frac{-2\sin(\pi\alpha)}{\pi} <0.$$

\end{proof}

From Lemmas \ref{cl:lambda1} and \ref{cl:nonempty} it follows that the interval of existence given by (\ref{eq:intervalo}), that is, $ \left( n(n-2)/4 +L^*, n(n-2)/4 +L_1\right),$ is nonempty if $2<n<4.$


\section{Existence of solutions}\label{sec:existencia}

In this section we present the proof of the following theorem: 


\begin{theorem}\label{teo:existencia}
For any $2<n<4$ and $\theta_1\in(0,\infty),$ the boundary value problem
\begin{equation}\label{eq:exi}
-u''(\theta) - (n-1)\coth \theta\, u'(\theta) = \lambda u+ u^{\frac{n+2}{n-2}}
\end{equation}
with $u\in H_0^1(\Omega),$ $u'(0) = u(\theta_1)=0,$ and $\theta\in  [0,\theta_1]$ has a positive solution if
\begin{equation*}
   \lambda\in \left( \frac{n(n-2)}{4} +L^*\,,\, \lambda_1\right).
   \end{equation*}
   
   \noindent Here $L^*$ is as in Definition \ref{def:losl}.
\end{theorem}


\bigskip

For natural values of $n,$ the positive solutions of

$$ -\Delta_{\mathbb{H}^n} u  = \lambda u + u^p,$$ 

\noindent on a geodesic ball with Dirichlet boundary conditions correspond to minimizers of 

$$ Q_{\lambda}(u)  =\frac{\displaystyle\int  |\nabla u|^2 \rho^{n-2}\, dx - \lambda \displaystyle\int u^2\rho^n\, dx }{\left(  \displaystyle\int u^\frac{2n}{n-2}\rho^n\, dx\right)^\frac{n-2}{n}}.$$

\noindent Here $\rho(x) = \dfrac{2}{1-|x|^2}$ is such that $ds= \rho \, dx.$ 

\bigskip

If $u$ is radial, we can write 

\begin{equation}\label{eq:cuociente}
 Q_{\lambda}(u)  =\frac{\omega_n\displaystyle\int_0^R  u'^2 \rho^{n-2}r^{n-1}\, dr - \lambda \omega_n \displaystyle\int_0^R u^2\rho^n r^{n-1}\, dr }{\left(  \omega_n \displaystyle\int_0^R u^\frac{2n}{n-2}\rho^n r^{n-1}\, dr\right)^\frac{n-2}{n}}.
 \end{equation}

\noindent Here $r = \tanh\left( \theta/2\right),$ $R = \tanh\left( \theta_1/2\right) <1,$ and $\omega_n$ represents the surface area of the unit sphere in $n$-dimensions, and is explicitly given by $\omega_n = 2\pi^\frac{n}{2}/\Gamma(n/2).$ This quotient is well defined if $n$ is a real number instead of a natural number.

\bigskip


\begin{lemma}\label{lem:primeraparte}
There exists a function $u\in H_0^1(\Omega),$ with $u'(0) = u(\theta_1)=0,$ such that $Q_{\lambda}(u) < S_n$ for all $\lambda >  \dfrac{n(n-2)}{4} +L^*.$ Here $S_n$ is the Sobolev constant. 
\end{lemma}
\begin{proof}
Let $\varphi$ be an arbitrary cutoff function such that $\varphi(0)=1,$ $\varphi'(0)=0$ and $\varphi(R)=0,$ and let

$$ v_{\epsilon}(r)  = \frac{\varphi(r)}{ (\epsilon + r^2)^\frac{n-2}{2}}.$$ 

\noindent As in \cite{St02}, let

$$ u_{\epsilon}(r) = \rho^{\frac{2-n}{2}}(r)v_{\epsilon}(r).$$ 

\noindent With this choice of $u_{\epsilon},$ and after integrating by parts, we have 

\begin{equation}\label{eq:larga}\begin{split}
\displaystyle\int_0^R  u'^2 \rho^{n-2}r^{n-1}\, dr = &\, \frac{n(n-2)}{4}\displaystyle\int_0^R \rho^2 v_{\epsilon}^2 r^{n+1}\, dr+ \frac{n(n-2)}{2} \displaystyle\int_0^R v_{\epsilon}^2\rho r^{n-1}\, dr \\+&\, \displaystyle\int_0^R v_{\epsilon}'^2 r^{n-1}\, dr. \\
\end{split}\end{equation}

\noindent Using the fact that $r^2 + \dfrac{2}{\rho} = 1$ to combine the first two terms of equation (\ref{eq:larga}), it follows that, 

\begin{equation}\label{eq:qlambda}
 Q_{\lambda}(u_{\epsilon}) = \frac{   \omega_n\left( \frac{n(n-2)}{4}-\lambda\right)   \displaystyle\int_0^R v_{\epsilon}^2 \rho^2 r^{n-1}\, dr + \omega_n\displaystyle\int_0^R v_{\epsilon}'^2 r^{n-1}\, dr   }{ \left( \omega_n \displaystyle\int_0^R v_{\epsilon}^\frac{2n}{n-2} r^{n-1}\, dr\right)^\frac{n-2}{2}}.
 \end{equation}


\begin{claim}

\begin{equation*}\begin{split}
 \omega_n\left( \frac{n(n-2)}{4}-\lambda\right)   \displaystyle\int_0^R v_{\epsilon}^2 \rho^2 r^{n-1}\, dr =  &\, \omega_n \left( \frac{ n(n-2)}{4}-\lambda\right) \displaystyle\int_0^R \varphi^2 r^{3-n}\rho^2\, dr\\ &\,+\mathcal{O}\left( \epsilon^\frac{4-n}{2}\right).
 \end{split}\end{equation*}

\end{claim}

\begin{proof}

Let 

$$I(\epsilon) = \int_0^R v_{\epsilon}^2\rho^2r^{n-1}\, dr = \int_0^R \frac{\varphi^2}{(\epsilon+r^2)^{n-2}}\rho^2r^{n-1}\, dr.$$

\noindent Then $I(0) = \displaystyle\int_0^R \varphi^2 \rho^2 r^{3-n}\, dr.$ Thus, it suffices to show that $ |I(\epsilon)-I(0)| = \mathcal{O}\left( \epsilon^\frac{4-n}{2}\right).$

\bigskip

If $0<r<R<1,$ then $\rho(r) = \dfrac{2}{1-r^2} < \dfrac{2}{1-R^2}.$ Thus, 

\begin{equation*}\begin{split}
|I(\epsilon) - I(0)| \le & \, \frac{4}{(1-R^2)^2} \left\vert \displaystyle\int_0^R \varphi^2r^{n-1} \left( \frac{1}{(\epsilon+r^2)^{n-2}} -\frac{1}{r^{2(n-2)}}\right)\, dr\right\vert \\=&\, \frac{4(n-2)}{(1-R^2)^2} \left\vert \displaystyle\int_0^R \displaystyle\int_0^\epsilon \frac{\left(\varphi^2-1+1\right)r^{n-1}}{(a+r^2)^{n-1}} \, da\,dr\right\vert. \\
\end{split}\end{equation*}

\noindent Let

$$ L_1(\epsilon) = \int_0^{\epsilon} \left(  \displaystyle\int_0^R  \frac{r^{n-1}}{(a+r^2)^{n-1}}\,dr \right) \, da, $$

\noindent and

$$ L_2(\epsilon) = \displaystyle\int_0^R (\varphi^2-1) r^{n-1}\int_0^{\epsilon} \frac{1}{(a+r^2)^{n-1}}\, da\,dr. $$

\noindent Making the change of variables $ r= u\sqrt{a}$ in the inner integral of $L_1(\epsilon)$, we have

$$ \displaystyle\int_0^R  \frac{r^{n-1}}{(a+r^2)^{n-1}}\,dr  = a^\frac{2-n}{2} \int_0^{\frac{R}{\sqrt{a}}} \frac{u^{n-1}}{(1+u^2)^{n-1}}\, du\le a^\frac{2-n}{2} \int_0^{\infty} \frac{u^{n-1}}{(1+u^2)^{n-1}}\, du.$$

\noindent Since we are considering $n > 2,$ this last integral converges. Thus, and since $n<4,$ 

$$ L_1(\epsilon) \le C\int_0^{\epsilon} a^\frac{2-n}{2} da  = \mathcal{O}\left( \epsilon^\frac{4-n}{2}\right). $$

On the other hand, since $\varphi (0)=1$ and $\varphi '(0) = 0,$ for $0\le r < 1$ we have that $\varphi^2-1 \le Cr^2$ for some $C>0.$ Thus, 

\begin{equation*}\begin{split}
 L_2(\epsilon) \le & C \displaystyle\int_0^R  r^{n+1}\int_0^{\epsilon} \frac{1}{(a+r^2)^{n-1}}\, da\,dr \\
 \le &  C \displaystyle\int_0^R  r^{n+1}\int_0^{\epsilon} \frac{1}{r^{2n-2}}\, da\,dr = C\epsilon \displaystyle\int_0^Rr^{3-n}\, dr.\\
 \end{split}\end{equation*}
 
 \noindent Since $n<4,$ this last integral converges. Thus $L_2(\epsilon) = \mathcal{O}(\epsilon),$ and in particular $\mathcal{O}(\epsilon^\frac{4-n}{2}).$

\end{proof}


\begin{claim}\label{cl:primera}

\begin{equation*}
\omega_n \displaystyle\int_0^R v_{\epsilon}'^2 r^{n-1}\, dr = \omega_n \int_0^R \varphi '(r)^2r^{3-n}\,dr +{K_1}{\epsilon^\frac{2-n}{2}} + \mathcal{O}(\epsilon^\frac{4-n}{2}),
\end{equation*}

\noindent where

\begin{equation*}
K_1 =\frac{\pi^\frac{n}{2} n(n-2)\Gamma\left(\frac{n}{2}\right)}{\Gamma(n)}. 
\end{equation*}

\end{claim}


\begin{proof}

Let 
$$ J= \omega_n \displaystyle\int_0^R v_{\epsilon}'^2 r^{n-1}\, dr.$$

\noindent Then we can write

$$J = \omega_n \displaystyle\int_0^R r^{n-1} \left[ \frac{\varphi '^2}{(\epsilon+r^2)^{n-2}} - \frac{2(n-2)r\varphi \varphi '}{(\epsilon+r^2)^{n-1}} + \frac{r^2\varphi ^2(n-2)^2}{(\epsilon+r^2)^n} \right] \, dr.$$

\noindent Integrating by parts the second term, and since by hypothesis $\varphi (R)=0,$ we have 

\begin{equation*}\begin{split}
J = & \omega_n \displaystyle\int_0^R \frac{\varphi '^2r^{n-1}}{(\epsilon+r^2)^{n-2}}\, dr + \omega_nn(n-2)\displaystyle\int_0^R \frac{\varphi ^2r^{n-1}}{(\epsilon+r^2)^{n-1}} \, dr \\ & - 2\omega_n(n-2)(n-1)\displaystyle\int_0^R \frac{\varphi ^2r^{n+1}}{(\epsilon + r^2)^n}\, dr +\omega_n(n-2)^2 \displaystyle\int_0^R \frac{\varphi ^2r^{n+1}}{(\epsilon+r^2)^n}\, dr. \\
\end{split}\end{equation*}

\noindent Thus, since $(n-2)^2-2(n-2)(n-1) = -n(n-2),$ combining the last three terms we have

\begin{equation}\label{eq:antesdeI}
J =  \omega_n \displaystyle\int_0^R \frac{\varphi '^2r^{n-1}}{(\epsilon+r^2)^{n-2}}\, dr + \omega_nn(n-2) \epsilon \displaystyle\int_0^R \frac{\varphi ^2r^{n-1}}{(\epsilon+r^2)^{n}} \, dr.
\end{equation}

\bigskip

Let us now estimate
\begin{equation*}
J_1(\epsilon) \equiv \int_0^R \varphi'(r)^2 (\epsilon + r^2)^{2-n} r^{n-1} \, dr.
\end{equation*}

\noindent Notice that 
\begin{equation*}
J_1(0) = \int_0^R \varphi'(r)^2 r^{3-n} \, dr.
\end{equation*}
In what follows we estimate the difference, i.e., $\Delta(\epsilon)\equiv J_1(\epsilon)-J_1(0)$. 
We write, 
$$
\Delta(\epsilon) = \int_0^1 \varphi'(r)^2 r^{3-n} (-A) \, dr, 
$$
where 
$$
A = 1-(\epsilon + r^2)^{2-n} \, r^{2n-4} = 1 - (1+\epsilon r^{-2})^{2-n} > 0,
$$
since $n>2.$ Using the fact that 
\begin{equation*}
(1+x)^{-m}> 1- m x
\end{equation*}
for $x =\epsilon/r^2 \ge 0$ and $m=n-2>0,$ we conclude that
$$
A < (n-2) \, \epsilon \, r^{-2}.
$$
Thus, 
\begin{equation}
|\Delta(\epsilon)| < \epsilon (n-2) \int_0^R \varphi'(r)^2 r^{1-n} \, dr. 
 \label{eq:es4}
 \end{equation}
 
 Notice that the integral on equation (\ref{eq:es4}) converges. In fact, since $\varphi (0)=1$ and $\varphi '(0) = 0,$ for $0\le r < 1$ one has $\varphi'(r)^2 \le C^2 r^2$ for some positive constant $C;$ thus $
 \varphi'(r)^2 r^{1-n} \le C  r^{3-n},$ which is integrable near $0$ for all $2 < n <4$. Hence $|\Delta(\epsilon)| = \mathcal{O}(\epsilon).$ Thus, from equation (\ref{eq:antesdeI}) we have 
 
 \begin{equation}\label{eq:despuesdeI}
J =  \omega_n \int_0^R \varphi '^2r^{3-n}\, dr + \omega_nn(n-2) \epsilon \int_0^R \frac{\varphi ^2r^{n-1}}{(\epsilon+r^2)^{n}} \, dr +\mathcal{O}(\epsilon).
\end{equation}

 Now let 

\begin{equation*}
J_2(\epsilon) \equiv \int_0^R \frac{\left(\varphi ^2-1\right)r^{n-1}+r^{n-1}}{(\epsilon+r^2)^{n}} \, dr. 
\end{equation*}

\noindent Making the change of variables $r=s\sqrt{\epsilon},$ we have 

\begin{equation*}
 \int_0^R \frac{r^{n-1}}{(\epsilon+r^2)^{n}} \, dr =\epsilon^\frac{-n}{2}\left(\int_0^{\infty} \frac{s^{n-1}}{(1+s^2)^n}\, ds - \int_{\frac{R}{\sqrt{\epsilon}}}^{\infty} \frac{s^{n-1}}{(1+s^2)^n}\, ds \right).
 \end{equation*}

\noindent But

\begin{equation*}
 \int_{\frac{R}{\sqrt{\epsilon}}}^{\infty} \frac{s^{n-1}}{(1+s^2)^n}\, ds \le  \int_{\frac{R}{\sqrt{\epsilon}}}^{\infty} s^{-n-1}\,ds = \frac{\epsilon^{\frac{n}{2}}}{nR^n}.
 \end{equation*}

 \noindent Notice that making the change of variables $u = s^2,$ we can write
 
 \begin{equation*}
  \int_0^{\infty} \frac{s^{n-1}}{(1+s^2)^n}\, ds = \frac{1}{2} \int_0^{\infty} \frac{u^{\frac{n}{2}-1}}{(1+u)^n}\,du = \frac{1}{2} \frac{\Gamma\left(\frac{n}{2}\right)^2}{\Gamma(n)}.
 \end{equation*}
 
 \noindent Here we have used the standard integral 
 
 \begin{equation*}
 \int_0^{\infty} \frac{x^{k-1}}{(1+x)^{k+m}}\, dx  = \frac{\Gamma(k)\Gamma(m)}{\Gamma(k+m)}
 \end{equation*}
 
 \noindent (see, e.g., \cite{Dw61}, equation 856.11, page 213), which holds for all $m,$ $k\,>0.$  Thus, 
 
 \begin{equation}\label{eq:repetida}
 \int_0^R \frac{r^{n-1}}{(\epsilon+r^2)^{n}} \, dr =  \frac{\Gamma\left(\frac{n}{2}\right)^2\epsilon^\frac{-n}{2}}{2\Gamma(n)} +\mathcal{O}(1).
 \end{equation}

 \bigskip

 On the other hand, since $\varphi^2(r) \le1+Cr^2,$ and setting once more $r=s\sqrt{\epsilon},$ we have that 

 

 \begin{equation*}
 \int_0^R \frac{(\varphi ^2-1)r^{n-1}}{(\epsilon+r^2)^{n}} \, dr \le C \epsilon^{\frac{2-n}{2}}\left(  \int_0^{\infty} \frac{s^{n+1}}{(1+s^2)^n}\, ds  -   \int_{\frac{R}{\sqrt{\epsilon}}}^{\infty} \frac{s^{n+1}}{(1+s^2)^n}\, ds \right). 
 \end{equation*}
 
 \noindent But 
 
 \begin{equation*}
  \int_{\frac{R}{\sqrt{\epsilon}}}^{\infty} \frac{s^{n+1}}{(1+s^2)^n}\, ds \le  \int_{\frac{R}{\sqrt{\epsilon}}}^{\infty} s^{1-n}\, ds = \mathcal{O}\left(\epsilon^{\frac{n-2}{2}}\right),
 \end{equation*}
 
 \noindent and $ \displaystyle\int_0^{\infty} \frac{s^{n+1}}{(1+s^2)^n}\, ds$ is finite. Thus, and since $2<n<4,$
 
  \begin{equation}\label{eq:ii22final}
  \int_0^R \frac{(\varphi ^2-1)r^{n-1}}{(\epsilon+r^2)^{n}} \, dr \le   C\int_0^R \frac{r^{n+1}}{(\epsilon+r^2)^{n}} \, dr =  \mathcal{O}( \epsilon^{\frac{2-n}{2}}).
  \end{equation}

\noindent  Therefore, from equations (\ref{eq:repetida}) and (\ref{eq:ii22final}) it follows that 
   
   \begin{equation*}
   J_2(\epsilon)  = \frac{\Gamma\left(\frac{n}{2}\right)^2\epsilon^{\frac{-n}{2}}}{2\Gamma(n)} + \mathcal{O}( \epsilon^{\frac{2-n}{2}}).  \end{equation*}
 
Finally, from equation (\ref{eq:despuesdeI}) it follows that 

 \begin{equation*}
J =  \omega_n \int_0^1 \varphi '^2r^{3-n}\, dr + \omega_nn(n-2) \epsilon^\frac{2-n}{2}\left( \frac{\Gamma\left(\frac{n}{2}\right)^2}{2\Gamma(n)}\right) +\mathcal{O}(\epsilon^\frac{4-n}{2}).
\end{equation*}

\noindent But we are taking $\omega_n =\frac{2\pi^\frac{n}{2}}{\Gamma\left(\frac{n}{2}\right)}.$ Thus, 

\begin{equation*}
J =  \omega_n \int_0^1 \varphi '^2r^{3-n}\, dr +  \epsilon^\frac{2-n}{2}\left( \frac{ n(n-2)\pi^\frac{n}{2}\Gamma\left(\frac{n}{2}\right)}{\Gamma(n)}\right) +\mathcal{O}(\epsilon^\frac{4-n}{2}).
\end{equation*}

 \end{proof}


 \begin{claim}\label{cl:segunda}
 
 \begin{equation*}
  \left( \omega_n \displaystyle\int_0^R v_{\epsilon}^\frac{2n}{n-2} r^{n-1}\, dr\right)^\frac{n-2}{n}= \epsilon^\frac{2-n}{2} K_2 + \mathcal{O}(\epsilon^\frac{4-n}{2}),
\end{equation*}
\noindent where 

\begin{equation*}
K_2  = \left( \pi^{n/2} \frac{\Gamma(n/2)}{\Gamma(n)}\right)^\frac{n-2}{n}.
\end{equation*}

 \end{claim}
 
 
 \begin{proof}
 
 Let 
 \begin{equation*}
H(\epsilon) \equiv \omega_n \displaystyle\int_0^R v_{\epsilon}^\frac{2n}{n-2} r^{n-1}\, dr = \omega_n \displaystyle\int_0^R \frac{\varphi(r)^{2n/(n-2)}}{(\epsilon + r^2)^{n}} \, r^{n-1} dr.
\end{equation*}

\noindent  Since $\varphi(0)=1$, this integral diverges when $\epsilon \to 0$. Denote by $H_1$ the leading behavior of $H(\epsilon)$, that is, 
 \begin{equation*}
 H_1(\epsilon) = \omega_n \int_0^R \frac{r^{n-1}}{(\epsilon + r^2)^{n}} \,  dr. 
 \end{equation*}

\noindent As in equation (\ref{eq:repetida}), we have 
 \begin{equation}
 H_1(\epsilon) =c_n \epsilon^{-n/2} + O(1), 
 \label{eq:es12}
 \end{equation}
where 
\begin{equation}
c_n  = \frac{\omega_n}{2} \frac{\Gamma(n/2)^2}{\Gamma(n)} = \pi^{n/2} \frac{\Gamma(n/2)}{\Gamma(n)}.
\label{eq:es13}
\end{equation}
\bigskip
 
It suffices now to show that 
\begin{equation*}
H(\epsilon) - H_1(\epsilon) = \omega_n \int_0^R \frac{\varphi(r)^{2n/(n-2)}-1}{(\epsilon + r^2)^{n}} \, r^{n-1} dr =\mathcal{O}\left(\epsilon^\frac{2-n}{2}\right).
\end{equation*}

\noindent But since $\varphi(r) \le 1+ C \, r^2$ for some positive constant $C,$ then 
\begin{equation}
|H(\epsilon) - H_1(\epsilon) |  \le C_n \int_0^R \frac{r^{n+1}}{(\epsilon + r^2)^{n}} \,  dr = \mathcal{O}\left(\epsilon^\frac{2-n}{2}\right),
\label{eq:es15}
\end{equation}

\noindent where the last equality follows from equation (\ref{eq:ii22final}). Thus, from (\ref{eq:es12}) and (\ref{eq:es15}), we conclude that
\begin{equation*}
H(\epsilon) = \epsilon^{-n/2} [c_n + O(\epsilon)], 
\end{equation*}
where $c_n$ is given by (\ref{eq:es13}).
 
 \end{proof}


Replacing the estimates obtained in the three previous claims in the definition of $Q_{\lambda}(u_{\epsilon})$ given in equation (\ref{eq:qlambda}), we obtain 

\begin{equation*}\begin{split}
Q_{\lambda}(u_{\epsilon}) =&\, \frac{K_1}{K_2} + \frac{\epsilon^\frac{n-2}{2}\omega_n}{K_2} \left(\left( \frac{n(n-2)}{4}-\lambda\right) \displaystyle \int_0^R \varphi^2 r^{3-n}\rho^2\, dr +\displaystyle\int_0^R \varphi'^2 r^{3-n}\, dr \right) \\ &\,+ \mathcal{O}(\epsilon).\\
\end{split}\end{equation*}

\noindent Here 

$$
K_1 =\frac{\pi^\frac{n}{2} n(n-2)\Gamma\left(\frac{n}{2}\right)}{\Gamma(n)},
$$

\noindent and 

$$
K_2  = \left( \pi^{n/2} \frac{\Gamma(n/2)}{\Gamma(n)}\right)^\frac{n-2}{n}.
$$

\noindent But 

$$ \frac{K_1}{K_2} = \pi n (n-2) \left( \frac{\Gamma\left(\frac{n}{2}\right)}{\Gamma(n)}\right)^\frac{2}{n},$$

\noindent which is precisely the Sobolev critical constant $S_n$ (see, e.g., \cite{Ta76}, with $p=2,$ $m=n$ and $q= \frac{2n}{n-2}$). Therefore, to conclude that $Q_{\lambda}(u_{\epsilon}) <S_n,$ it suffices to show that for $\lambda > n(n-2)/4 +L^*,$ there exists a choice of $\varphi$ such that

$$ F(\varphi) \equiv \left(\frac{n(n-2)}{4}-\lambda\right) \displaystyle \int_0^R \varphi^2 r^{3-n}\rho^2\, dr +\displaystyle\int_0^R \varphi'^2 r^{3-n}\, dr $$

\noindent is negative. 

\bigskip

Let 

$$ M(\vf) = \displaystyle\int_0^R \varphi'^2 r^{3-n}\, dr ,$$

\noindent and let $\vf_1$ be the minimizer of $M(\vf)$ subject to the constraint  $\displaystyle \int_0^R \varphi^2 r^{3-n}\rho^2\, dr =1.$ Then $\vf_1$ satisfies the Euler equation 

\begin{equation}\label{eq:ecneuler}
 -\left( \vf_1'r^{3-n}\right)' = \mu \vf_1r^{3-n}\rho^2.
 \end{equation}

 \noindent Here $\mu$ is the Lagrange multiplier. Multiplying equation (\ref{eq:ecneuler}) by $\varphi_1$ and integrating this equation by parts, and since $\displaystyle \int_0^R \varphi_1^2 r^{3-n}\rho^2\, dr =1,$ we obtain 
 
 \begin{equation}\label{eq:mupos}
  \displaystyle \int_0^R \vf_1'^2r^{3-n}\, dr = \mu.
  \end{equation}
 
 \noindent It follows that $F(\vf_1) = \dfrac{n(n-2)}{4}-\lambda  +\mu.$ Thus, $F(\vf_1)$ is negative as long as $\lambda > \dfrac{n(n-2)}{4}+\mu.$  Notice that from (\ref{eq:mupos}) one has that $\mu$ is positive. 
 \bigskip

 It suffices now to show that $\mu = L^*.$ Multiplying equation (\ref{eq:ecneuler}) by $-r^{n-3},$ we obtain 
 
 \begin{equation}\label{eq:estrella}
 \vf'' + \frac{(3-n)}{r}\vf' + \mu\varphi\rho^2 = 0.
 \end{equation}
 
 \noindent Making the change of variables $\vf(r) = r^\frac{n-2}{2}v(r),$ and after some rearrangement of terms, we can write equation (\ref{eq:estrella}) as 
 
 \begin{equation}\label{eq:leg1}
  v'' + \frac{v'}{r} + \left( \mu\rho^2 - \frac{(n-2)^2}{4r^2}\right) v =0. 
 \end{equation}
 
Changing back to geodesic coordinates, and since $r = \tanh \frac{\theta}{2},$ we can rewrite equation (\ref{eq:leg1}) as 

\begin{equation}\label{eq:leg2}
 v''  +  \coth\theta  v'  + \left(\mu - \frac{\alpha^2}{\sinh^2\theta} \right)v = 0,
\end{equation}
 
 \noindent where $\alpha = \frac{2-n}{2}.$ Equation (\ref{eq:leg2}) is a Legendre equation, whose solutions are $P_{\ell}^{\alpha}$ and $P_{\ell}^{-\alpha},$ where $-\ell(\ell+1) = \mu.$ It follows from equation (\ref {eq:expansionenseries}) that the regular solution to equation (\ref{eq:estrella}) is  
 $$\vf(\theta) = \tanh^{-\alpha} \left(\dfrac{\theta}{2}\right) P_{\ell}^{-\alpha}(\cosh\theta).$$
 
 \noindent Since the solution must vanish at the boundary, it follows that $L = L^*.$ Thus, $\mu = L^*.$ This finishes the proof of Lemma \ref{lem:primeraparte}.

 \end{proof}

 
 The proof of Theorem \ref{teo:existencia} now follows easily from a result by Lieb. In fact, by Lemma 1.2 in \cite{BrNi83}, it follows that if there exists some $u$ such that $Q_{\lambda}(u)<S_n,$ then there exists a minimizer of $Q_{\lambda}.$ Given any constant $\eta>0,$ the quotient $Q_{\lambda}(u)$ is invariant under the transformation $u\rightarrow \eta u.$ In order to compute the corresponding Euler equation, we minimize the numerator of equation (\ref{eq:cuociente}) subject to the constraint $ \omega_n \displaystyle\int_0^R u^\frac{2n}{n-2}\rho^n r^{n-1}\, dr=1.$ We obtain 
 
 \begin{equation}\label{ec:euler2}
 \left( u'\rho^{n-2}r^{n-1}\right)' + \lambda u \rho^n r^{n-1} + \eta u^p \rho^n r^{n-1}=0,
 \end{equation}
 
 \noindent where $\eta$ is a Lagrange multiplier. Multiplying through by $\omega_n u,$ integrating between $0$ and $R,$ and integrating by parts, we obtain 
 
 \begin{equation*}\begin{split}
 \eta &\, = \omega_n \left(\int_0^R u'^2 \rho^{n-2} r^{n-1}\, dr - \lambda  \int_0^R u^2 \rho^n r^{n-1}\, dr \right) \\&\, \ge (\lambda_1-\lambda) \omega_n \int_0^R u^2 \rho^n r^{n-1}\, dr.\\
 \end{split}\end{equation*}
 
 \noindent This last inequality follows from the variational characterization of $\lambda_1.$  It follows that $\eta > 0$ provided that $\lambda<\lambda_1.$ Setting $u = \eta^\frac{-1}{p-1}v$ in (\ref{ec:euler2}) one has that $v$ satisfies 
 
 \begin{equation}\label{eq:penultima}
   \left( u'\rho^{n-2}r^{n-1}\right)' + \lambda u \rho^n r^{n-1} +  u^p \rho^n r^{n-1}=0.
   \end{equation}
 
 \noindent Finally, setting $r =\tanh \frac{\theta}{2},$ equation (\ref{eq:penultima}) becomes (\ref{eq:exi}). This finishes the proof of Theorem \ref{teo:existencia}.


\section{Nonexistence of solutions} \label{sec:noexistencia}

In this section we use a Pohozaev type argument to show that if $2<n<4$ then problem (\ref {eq:ode}) has a solution gap.

\begin{theorem}\label{lem:noexistencia}

For any $2<n<4$ and $\theta_1\in(0,\infty),$ the boundary value problem
\begin{equation}\label{ecn:1}
-u''(\theta) - (n-1)\coth \theta\, u'(\theta) = \lambda u+ u^{\frac{n+2}{n-2}}
\end{equation}
\noindent with $u\in H_0^1(\Omega),$ $u'(0) = u(\theta_1)=0,$ and $\theta\in  [0,\theta_1],$ has no solution if
\begin{equation}
   \lambda\in \left(\frac{n(n-2)}{4}\,,\, \frac{n(n-2)}{4} +L^* \right].
   \end{equation}
   
   \noindent Here $L^*$ is as in Definition \ref{def:losl}. 

\end{theorem}

\begin{proof}
Let $g$ be a smooth nonnegative function such that $g(0) = g'(0)=0.$ Writing equation (\ref{ecn:1}) as

\begin{equation}\label{ecn:2}
\frac{-(\sinh^{n-1}\theta\, u')'}{\sinh^{n-1}\theta} = \lambda u + u^p,
\end{equation}

\noindent multiplying through by  $g(\theta)u'(\theta)\sinh^{2n-2}\theta,$ and integrating, we obtain

\begin{equation*}\begin{split}
- \int_0^{\theta_1} \left( \frac{(\sinh^{n-1}\theta\, u')^2}{2}\right)'g\, d\theta  = &\,   \lambda \int_0^{\theta_1}\left(\frac{u^2}{2}\right)'g\sinh^{2n-2}\theta\, d\theta \\ &\,
+\int_0^{\theta_1}\left(\frac{u^{p+1}}{p+1}\right)'g\sinh^{2n-2}\theta\,d\theta.\\
\end{split}\end{equation*}

Integrating by parts, and since $u(\theta_1) = 0,$ we obtain 

\begin{equation}\label{ecn:intpp1}\begin{split}
&\frac{1}{2}\int_0^{\theta_1}u'^2  g'\sinh^{2n-2} \,d\theta + \frac{\lambda}{2} \int_0^{\theta_1} u^2(g\sinh^{2n-2}\theta)'\,d\theta \\ &+ \int_0^{\theta_1} \frac{u^{p+1}}{p+1}(g\sinh^{2n-2}\theta)'\,d\theta 
 =\, \frac{\sinh^{2n-2}\theta_1 (u'(\theta_1))^2g(\theta_1)}{2}.\\
\end{split}\end{equation}

Let $f(\theta) = \frac{1}{2}g'\sinh^{n-1}\theta.$ Multiplying equation (\ref{ecn:2}) by $f(\theta)u(\theta)\sinh^{n-1}\theta$ and integrating, we obtain

$$ - \int_0^{\theta_1}(\sinh^{n-1}\theta\, u')' fu\,d\theta = \lambda\int_0^{\theta_1}f\sinh^{n-1}\theta \,u^2\,d\theta +\int_0^{\theta_1} u^{p+1}f\sinh^{n-1}\theta\,d\theta.$$ 

\noindent After integrating by parts, this last equation can be written as 

\begin{equation}\label{ecn:intpp2}\begin{split}
\int_0^{\theta_1} u'^2 f\sinh^{n-1}\theta \, d\theta = &\, \int_0^{\theta_1} u^2 \left( \lambda f\sinh^{n-1}\theta + \frac{1}{2}(f'\sinh^{n-1}\theta)'\right)\, d\theta \\  &\, +\int_0^{\theta_1}u^{p+1}f\sinh^{n-1}\theta\, d\theta. \\
\end{split}\end{equation}

\noindent By subtracting equation (\ref{ecn:intpp1}) from equation (\ref{ecn:intpp2}) we obtain 

\begin{equation}\label{ecn:contradiccion1}
\int_0^{\theta_1} A(\theta)u(\theta)^2\,d\theta + \int_0^{\theta_1} B(\theta) u(\theta)^{p+1}\,d\theta =  \frac{\sinh^{2n-2}\theta_1 (u'(\theta_1))^2g(\theta_1)}{2},
\end{equation}

\noindent where

\begin{equation*}
A(\theta) \equiv \frac{1}{2}(f'(\theta)\sinh^{n-1}\theta)' + \lambda f(\theta)\sinh^{n-1}\theta + \frac{\lambda}{2}(g(\theta)\sinh^{2n-2}\theta)';
\end{equation*}

\noindent and

\begin{equation*}
 B(\theta) = f(\theta)\sinh^{n-1}\theta + \frac{(g(\theta)\sinh^{2n-2}(\theta))'}{p+1}. 
\end{equation*}

\noindent Notice that the right-hand side of equation (\ref{ecn:contradiccion1}) is nonnegative. We will show that the left-hand side of (\ref{ecn:contradiccion1}) is negative, thus arriving at a contradiction. 
\bigskip

Using the definition of $f$ and simplifying, we can write 

\begin{equation*}\begin{split}
 A(\theta) =&\, \sinh^{2n-2}\theta \left[  \frac{g'''}{4} + \frac{3}{4}(n-1)\coth\theta g''  \right.\\ &\left.+\left( \lambda   + \frac{n-1}{4} + \frac{(n-1)(2n-3)}{4}\coth^2\theta \right)g'  + \lambda (n-1)\coth\theta g \right].\\
 \end{split}\end{equation*}

Finally, making the change of variables $T (\theta)= g(\theta)\sinh^2\theta,$ we obtain 

\begin{equation*}\begin{split}
A(\theta)=&\,\sinh^{2n-4}\theta\left[  \frac{T'''}{4} +\frac{3}{4}(n-3)\coth\theta\, T'' + \left(  \frac{1}{4}\coth^2\theta(n-3)(2n-11) \right.\right. \\
&\left.\left. +\,\lambda +\frac{1}{4}(n-7)\right)T' + (n-3)\left(\coth\theta(\lambda-2)-\coth^3\theta (n-4)\right)\,T   \right].\\
\end{split}\end{equation*}

Simplifying $B,$ we obtain

\begin{equation*}
B(\theta) = \frac{(n-1)\sinh^{2n-2}\theta}{n}\left( g'(\theta) + (n-2)\coth\theta g\right).
\end{equation*}

\bigskip

\noindent As before, we make the change of variables $T (\theta)= g(\theta)\sinh^2\theta,$ to obtain

\begin{equation*}
B(\theta) = \frac{(n-1)}{n}\sinh^{2n-4}\theta \left( T' + (n-4)\coth\theta \,T\right).
\end{equation*}

We will show that there is a choice of $T$ for which $A(\theta) \equiv 0.$ We will then show that for this choice of $T,$ $B(\theta)$ is negative as long as 

\begin{equation}\label{eq:intervalono}
   \lambda\in \left(\frac{n(n-2)}{4}\,,\, \frac{n(n-2)}{4} +L^* \right].
   \end{equation}

\begin{lemma}
Consider the equation
\begin{equation}\label{eq:noexistencia}\begin{split}
&\frac{T'''}{4} +\frac{3}{4}(n-3)\coth\theta\, T'' + \left(  \frac{1}{4}\coth^2\theta(n-3)(2n-11) +\lambda +\frac{1}{4}(n-7)\right)T' \\
 &+ (n-3)\left(\coth\theta(\lambda-2)-\coth^3\theta (n-4)\right)\,T = 0.\\
\end{split}\end{equation}

Then 
$$T(\theta) = \sinh^{4-n}\theta P_{\ell}^{\alpha}(\cosh\theta)P_{\ell}^{-\alpha}(\cosh\theta)$$

\noindent  is a solution of (\ref{eq:noexistencia}), where $\alpha = (2-n)/{2}$ and $\ell(\ell+1) = \alpha(\alpha-1)-\lambda.$

\end{lemma}


\begin{proof}
Let $v(\theta) = y_1(\theta)y_2(\theta),$ where $y_1(\theta) = P_{\ell}^{\alpha}(\cosh\theta)$ and $y_2(\theta) = P_{\ell}^{-\alpha}(\cosh\theta).$ Then $y_1$ and $y_2$ are solutions of 

\begin{equation}\label{eq:legendre1}
y''(\theta) +\coth\theta  y'(\theta) + k(\theta) y(\theta) =0,
\end{equation}

\noindent where 

\begin{equation*}
k(\theta) = -\ell(\ell+1) - \frac{\nu^2}{\sinh^2\theta}. 
\end{equation*}

It follows from equation (\ref{eq:legendre1}) that 

$$  y_1'' \,y_2 +  y_2'' \,y_1  = -\coth\theta\,  v' -2k v,$$

\noindent and from the above that 

$$  v'' = 2 y_1' y_2' - \coth\theta  v' -2k v.$$ 

\noindent Similarly, we can write 

$$  y_1''  y_2' + y_1'  y_2''  = -2\coth\theta  y_1'  y_2' - k v',$$ 

\noindent from which it follows that 

$$ v''' = -\coth\theta  v'' + \left( \frac{1}{\sinh^2\theta}-4k\right)  v'  -2 k' v - 4\coth\theta  y_1' y_2'.$$

\noindent Using the fact that $ y_1' y_2' = \frac{1}{2}\left(  v'' +\coth\theta v' + 2kv\right),$ we obtain 

\begin{equation}\label{eq:noex3}
v''' + 3\coth\theta  v'' + \left( 2\coth^2\theta +4k - \frac{1}{\sinh^2\theta}\right) v' + \left( 2 k' + 4k\coth \theta\right) v = 0. 
\end{equation}

Finally, replacing $ v(\theta) = T(\theta) \sinh^{n-4}\theta $ in equation (\ref{eq:noex3}), using the fact that $\ell(\ell+1) =\alpha(\alpha-1)-\lambda,$ and after significant simplification and rearrangement of terms, we obtain precisely equation (\ref{eq:noexistencia}). 

\end{proof}


It suffices now to show that for $T$ as in the previous lemma, $B$ is negative. We do so in the following lemma. 



\begin{lemma} 

 Let $$T(\theta) = \sinh^{4-n}\theta P_{\ell}^{\alpha}(\cosh\theta)P_{\ell}^{-\alpha}(\cosh\theta)$$

\noindent  where $\alpha = (2-n)/{2},$ $\theta\in(0,\theta_1),$ and $L= -\ell(\ell+1) =\lambda- \alpha(\alpha-1).$ Then 
\begin{equation}\label{eq:antesdet}
B(\theta) = \frac{(n-1)}{n}\sinh^{2n-4}\theta \left( T' + (n-4)\coth\theta \,T\right) 
\end{equation}
\noindent is negative if $0<L\le L^*.$ 
\end{lemma}


\begin{proof}

Notice that the condition $0<L\le L^*$ is precisely the same as (\ref{eq:intervalono}).  Substituting $T(\theta) = \sinh^{4-n}\theta P_{\ell}^{\alpha}(\cosh\theta)P_{\ell}^{-\alpha}(\cosh\theta)$ in equation (\ref{eq:antesdet}), we obtain 
$$B(\theta) = \frac{(n-1)}{n}\sinh^{n+1}\theta \left( \dot P_{\ell}^\alpha P_{\ell}^{-\alpha} + P_{\ell}^{\alpha}\dot P_{\ell}^{-\alpha} \right). $$

\noindent Since $\sinh\theta$ is positive for $\theta>0,$ and since $P_{\ell}^{\alpha}P_{\ell}^{-\alpha}>0$ if $0<L\le L^*,$ it suffices to show that

$$ \frac{\dot P_{\ell}^\alpha}{P_{\ell}^\alpha} + \frac{\dot P_{\ell}^{-\alpha}}{P_{\ell}^{-\alpha}} <0.$$ 

Let 

\begin{equation}\label{eq:defdey}
 y_{\nu}(\theta) = \frac{1}{\sinh \theta} \frac{P_{\ell}^{\nu+1}}{P_{\ell}^{\nu}} + \frac{ \nu}{2\sinh^2\frac{\theta}{2}}.
 \end{equation}

\noindent Then, by the raising relation given by equation (\ref{eq:subida}) it follows that 

$$ \frac{\dot P_{\ell}^\alpha}{P_{\ell}^\alpha} + \frac{\dot P_{\ell}^{-\alpha}}{P_{\ell}^{-\alpha}} = \frac{1}{\sinh\theta}\left( \frac{P_{\ell}^{\alpha+1}}{P_{\ell}^\alpha} + \frac{P_{\ell}^{-\alpha+1}}{P_{\ell}^{-\alpha}}\right) = y_{\alpha} + y_{-\alpha}.$$ 

\noindent We will show that for $\theta \in (0,\theta_1),$ and if $-1<\nu<1,$ then $ y_{\nu}(\theta)<0.$ This will imply that $y_{\alpha}(\theta) + y_{-\alpha}(\theta) <0,$ and therefore  that $B$ is negative. 

\bigskip 

From equations (\ref{eq:expansionenseries}) and (\ref{eq:hg}) it follows that 

$$ P_{\ell}^{\nu} = \frac{1}{\Gamma(1-\nu)}\coth^\nu\left( \frac{\theta}{2}\right) \left(1+ \frac{\ell(\ell+1)}{1-\nu}\sinh^2\left( \frac{\theta}{2}\right) + \mathcal{O}\left( \sinh^4\left( \frac{\theta}{2}\right)\right) \right).$$

\noindent Then, and since $\Gamma(1-\nu) = -\nu\Gamma(-\nu),$ we can write 

$$ \frac{P_{\ell}^{\nu+1}}{P_{\ell}^\nu} = -\nu \coth\left( \frac{\theta}{2}\right) \left( 1 - \frac{\ell(\ell+1)}{\nu(1-\nu)}\sinh^2\left(\frac{\theta}{2}\right) + \mathcal{O}\left(\sinh^4  \left(\frac{\theta}{2}\right)\right)\right).$$

\noindent Therefore, and since $\coth\left(\frac{\theta}{2}\right)/\sinh\theta = \left(2\sinh^2\left(\frac{\theta}{2}\right)\right)^{-1},$ we have 

$$ y_{\nu} = \frac{\ell(\ell+1)}{2(1-\nu)} + \mathcal{O}\left(\sinh^2\left(\frac{\theta}{2}\right)\right).$$ 

\noindent Thus, if $-1<\nu<1,$ and since $\ell(\ell+1)<0,$

$$ \lim_{\theta\rightarrow 0}y_{\nu}(\theta) = \frac{\ell(\ell+1)}{2(1-\nu)}<0.$$ 

\noindent We will show by contradiction that there is no point at which $y_{\nu}$ changes sign, thus concluding that $y_{\nu}(\theta)$ is negative for all $\theta> 0.$

\bigskip

Taking the derivative of equation (\ref{eq:defdey}), we obtain 

$$ y_{\nu}' = -\frac{\cosh\theta}{\sinh^2\theta}\frac{P_{\ell}^{\nu+1}}{P_{\ell}^\nu} + \frac{\dot P_{\ell}^{\nu+1}}{P_{\ell}^\nu} - \frac{\dot P_{\ell}^\nu}{P_{\ell}^\nu} \frac{P_{\ell}^{\nu+1}}{P_{\ell}^\nu} - \frac{\nu}{2}\frac{\cosh \left(\frac{\theta}{2}\right)}{\sinh^3 \left(\frac{\theta}{2}\right)}.$$ 

\noindent Using the raising and lowering relations given in equations (\ref{eq:subida}) and (\ref{eq:bajada}), we can write 

\begin{equation*}\begin{split}
 y_{\nu}' =&\, \frac{-1}{\sinh\theta}\left( \frac{P_{\ell}^{\nu+1}}{P_{\ell}^\nu}\right)^2 + \frac{(-2\nu-2)\cosh\theta}{\sinh^2\theta}\left( \frac{P_{\ell}^{\nu+1}}{P_{\ell}^\nu}\right) \\&\,+ \frac{\ell(\ell+1)-\nu(\nu+1)}{\sinh\theta} - \frac{\nu\cosh\frac{\theta}{2}}{2\sinh^3\frac{\theta}{2}}.\\
 \end{split}\end{equation*} 

\noindent Solving for $\left( \dfrac{P_{\ell}^{\nu+1}}{P_{\ell}^\nu}\right)$ from equation (\ref{eq:defdey}), and after rearranging terms, we obtain 

\begin{equation}\label{eq:Ricatti}
 y_{\nu}'  = -\sinh\theta y_{\nu}^2 + \frac{2(\nu-\cosh\theta)}{\sinh\theta}y_{\nu} + \frac{\ell(\ell+1)}{\sinh\theta}.
 \end{equation}

Now suppose there was a point $\theta^*$ at which $y_{\nu}(\theta^*)$ crossed the $\theta$-axis.  At this point, we would have $y_{\nu}(\theta^*) = 0$ and $y_{\nu}'(\theta^*)>0.$ But evaluating equation (\ref{eq:Ricatti}) at $\theta^*,$ we obtain

$$ y_{\nu}'(\theta^*) = \frac{\ell(\ell+1)}{\sinh\theta^*} <0,$$

\noindent arriving at a contradiction. 

\end{proof}

This completes the proof of Theorem \ref{lem:noexistencia}.

\end{proof}


\section{Uniqueness}\label{sec:unicidad}

\begin{lemma}

The problem 

\begin{equation}\label{eq:uni1}u''(\theta) + (n-1)\coth(\theta)u'(\theta) +\lambda u(\theta) + u(\theta)^p = 0
\end{equation}

\noindent with $u'(0) = u(\theta_1)= 0,\,$ $2<n<4,\,$ and $\lambda > \dfrac{n(n-2)}{4},$ has at most one positive solution. 

\end{lemma}


\begin{proof}
The proof of this lemma follows directly from \cite{KwLi92}. In fact, making the change of variables $u\rightarrow v$ given by $u(\theta) = \sinh^\alpha(\theta)v(\theta),$ where $\alpha = \dfrac{2-n}{2},$ equation (\ref{eq:uni1}) can be written as

\begin{equation}\label{eq:uni2}
\sinh^2(\theta)v''(\theta) +\sinh\theta\cosh\theta v'(\theta) + G_{\lambda}(\theta)v(\theta) + v(\theta)^p=0,
\end{equation}

\noindent where 

\begin{equation*}
G_{\lambda}(\theta) = -\alpha^2 + \left[ \lambda- \frac{n(n-2)}{4}\right]\sinh^2\theta. 
\end{equation*}

We define the energy function

\begin{equation*}
E[v] \equiv  \sinh^2\theta v'(\theta)^2 + \frac{2}{p+1}v(\theta)^{p+1} + G_{\lambda}(\theta)v(\theta)^2 = 0. 
\end{equation*}

\noindent Then if $v(\theta)$ is a solution of (\ref{eq:uni2}), 

\begin{equation*}
\frac{dE}{d\theta} = G_{\lambda}'(\theta) v(\theta)^2.
\end{equation*}

The function $G_{\lambda}(\theta)$ is increasing as long as $\lambda > \dfrac{n(n-2)}{4}.$ That is, $G_{\lambda}(\theta)$ is a $\Lambda-function$ and it follows from \cite{KwLi92} that $v$ (and therefore $u$) is unique. 
\end{proof}

\begin{remark}

Uniqueness of solutions to this problem for $\lambda \in \left(n(n-2)/4, (n-1)^2/4\right]$ was obtained by Mancini and Sandeep (see Proposition 4.4 in \cite{MaSa08}). Notice that $\lambda = (n-1)^2/4$ corresponds to the first eigenvalue in the limiting case $\theta_1 = \infty.$ The interval considered in \cite{MaSa08} is a strict subinterval of the interval we consider here. 

\end{remark}

\bibliographystyle{spmpsci}      

\bibliography{monatbn}   

\end{document}